\documentclass[reqno, 11pt]{amsart}
\usepackage{amsfonts}
\usepackage{amsmath}
\usepackage{tikz-cd}
\usepackage{mathtools}
\usepackage{amssymb}
\usepackage{amsthm}
\usepackage{amscd}
\usepackage{color}
\usepackage{verbatim}

\usepackage[colorlinks]{hyperref}
\definecolor{linkblue}{RGB}{1,1,190}
\definecolor{citered}{RGB}{190,1,1}  
\hypersetup{
	linkcolor=linkblue,
	urlcolor=linkblue,
	citecolor=citered
}

\newtheorem{theorem}{Theorem}
\newtheorem{lemma}[theorem]{Lemma}

\newtheorem{corollary}[theorem]{Corollary}

\newcommand{\N}{\mathbb N}
\newcommand{\Z}{\mathbb Z}

\DeclareMathOperator{\fin}{fin}
 
\def\P{\ensuremath\mathcal{P}}
\def\Pf{\ensuremath\mathcal{P}_{\text{fin},1}}

\def\O{\ensuremath\mathcal{O}}

\begin{document}
\title[A counterexample to an isomorphism problem for power monoids]{A counterexample to an isomorphism problem \\ for power monoids}
\author{Balint Rago}
\address{University of Graz, NAWI Graz, Department of Mathematics and Scientific Computing, Heinrichstraße 36,
8010 Graz, Austria}

\thanks{This work was supported by the FWF (projects W1230 and  10.55776/DOC-183-N)}

\email{balint.rago@uni-graz.at}



\begin{abstract}
    Let $H$ be a (multiplicatively written) monoid. The family $\Pf(H)$ of finite subsets of $H$ containing the identity element is itself a monoid when endowed with setwise multiplication induced by $H$. Tringali and Yan proved that two monoids $H_1$ and $H_2$ contained in a special class of commutative, cancellative monoids are isomorphic if and only if $\Pf(H_1)$ and $\Pf(H_2)$ are. Moreover, they raised the question whether the same holds in the general setting of cancellative monoids. We show that if $H_1$ and $H_2$ are (commutative) valuation monoids with trivial unit groups and isomorphic quotient groups, then $\Pf(H_1)\simeq\Pf(H_2)$. This provides a negative answer to Tringali and Yans question already within the class of valuation submonoids of the additive group $\mathbb{Z}^2$. 
\end{abstract}

\maketitle

\section{Introduction}

Let $S$ be a (multiplicatively written) semigroup. We denote by $\P(S)$ the \textit{large power semigroup} of $S$, i.e.\ the family of all non-empty subsets of $S$, endowed with the binary operation of setwise multiplication \[(X,Y)\mapsto \{xy:x\in X,y\in Y\}.\] Moreover, for a monoid $H$ with identity $1$, we denote by $\Pf(H)$ the set of finite subsets of $H$ that contain $1$; it is a submonoid of $\P(H)$ with identity $\{1\}$, called the \textit{reduced finitary power monoid} of $H$. 

The systematic study of power semigroups began in the 1960s and was initiated by Tamura and Shafer. A central question that arose from their work, called the \emph{isomorphism problem for power semigroups}, is whether, for semigroups $S$ and $T$ in a certain class $\O$, an isomorphism between $\P(S)$ and $\P(T)$ implies that $S$ and $T$ are isomorphic. Although this was answered in the negative by Mogiljanskaja \cite{Mo73} for the class of all semigroups, several other classes have been found for which the answer is positive, see for example \cite{Ga-Zh14,Go-Is84,Sh67, Sh-Ta67,Tr25}. More recently, Tringali pushed forward the investigation of the arithmetic of reduced finitary power monoids \cite{An-Tr21,Co-Tr25,Fa-Tr18,Tr22}. 
In \cite{Bi-Ge25}, Bienvenu and Geroldinger conjectured that an analogue of the isomorphism problem holds for the reduced finitary power monoids of numerical monoids. This was soon resolved by Tringali and Yan \cite{Tr-Ya25}, who more generally proved that if $H_1$ and $H_2$ are submonoids of the non-negative rational numbers under addition, then $H_1$ and $H_2$ are isomorphic if and only if $\Pf(H_1)$ and $\Pf(H_2)$ are. Moreover, the authors showed that this does not hold true for arbitrary monoids \cite[Examples 1.2]{Tr-Ya25}, by constructing counterexamples involving non-cancellative monoids, leaving the question open whether the isomorphism problem for $\Pf(H)$ holds true in the ``well-behaved" setting of cancellative monoids. 

In the next section, we provide a negative answer to this question by first showing that the reduced finitary power monoids of two reduced valuation monoids with isomorphic quotient groups are always isomorphic. Using this result, we then present non-isomorphic submonoids of $(\mathbb{Z}^2,+)$ with isomorphic reduced finitary power monoids.

\section{The example}

We denote by $\N$ the set of positive integers and by $\N_0$ the set of non-negative integers. Let $G$ be an abelian group, $X\subseteq G$ a subset and $a\in G$. We write \[aX:=\{ax:x\in X\}\] and \[X^{-1}:=\{x^{-1}:x\in X\}.\]  

We say that a commutative monoid $H$ is \textit{cancellative} if for all $a,b,c\in H$ with $ac=bc$, it holds that $a=b$. For a commutative, cancellative monoid $H$, we denote by $\mathsf{q}(H)$ the \textit{quotient group} of $H$, that is, the unique abelian group up to isomorphism with the property that any abelian group containing an isomorphic image of $H$, also contains an isomorphic image of $\mathsf{q}(H)$. We have \[\mathsf{q}(H)=\{ab^{-1}:a,b\in H\}\] and throughout this paper, we will always assume that $H\subseteq \mathsf{q}(H)$.
We say that a commutative, cancellative monoid $H$ is a \textit{valuation monoid} if for every $x\in \mathsf{q}(H)$, we have $x\in H$ or $x^{-1}\in H$. Moreover, $H$ is called \textit{reduced} if the only invertible element of $H$ is the identity $1$. Lastly, we call a non-invertible element $a\in H$ \textit{irreducible} if it cannot be written as a non-trivial product of non-invertible elements of $H$.

\begin{lemma}\label{lemma}
    Let $H$ be a reduced valuation monoid with quotient group $G=\mathsf{q}(H)$ and let $X\in \P_{\fin,1}(G)$. There is a unique $a\in G$ such that $aX\in\P_{\fin,1}(H)$.
\end{lemma}

\begin{proof}
    Suppose first that there are distinct $a_1,a_2\in G$ such that $a_1X,a_2X\in\P_{\fin,1}(H)$. Then $a_1^{-1},a_2^{-1}\in X$, which implies that $a_1a_2^{-1}\in H$ and $a_2 a_1^{-1}\in H$, a contradiction to $H$ being reduced. Hence there is at most one $a\in G$ such that $aX\in\P_{\fin,1}(H)$. \\

    To show the existence of such an $a$, we proceed by induction on the cardinality of $X$. If $|X|=1$, we have $X=\{1\}$ and the assertion is clearly true. Suppose that $|X|\geq 2$, let $1\neq x\in X$ and set $Y:=X\setminus \{x\}$. By the induction hypothesis, there is $a\in G$ such that $aY\in \P_{\fin,1}(H)$. If $ax\in H$, we are done. Otherwise, we have $aY\subseteq H$, but $ax\not\in H$, whence \[
        (xY^{-1})\cap H=(ax(aY)^{-1})\cap H =\emptyset.
    \] However, since $H$ is a valuation monoid, this implies that $x^{-1}Y\subseteq H$ and consequently $x^{-1}X\in \P_{\fin,1}(H)$.
\end{proof}

\begin{theorem}\label{prop}
    Let $H_1,H_2$ be reduced valuation monoids with $\mathsf{q}(H_1)\simeq\mathsf{q}(H_2)$. Then $\P_{\fin,1}(H_1)\simeq \P_{\fin,1}(H_2)$.
\end{theorem}
\begin{proof}
   Set $G:=\mathsf{q}(H_1)$. Without loss of generality, we may replace $H_2$ with any isomorphic copy of itself. Since $G\simeq\mathsf{q}(H_2)$, we can thus assume that $H_2\subseteq G$ and $\mathsf{q}(H_2)=G$. We define an isomorphism $
    f:\P_{\fin,1}(H_1)\to \P_{\fin,1}(H_2)$ in the following way. 

    If $X\in \P_{\fin,1}(H_1)$, then $f(X)=aX$, where $a\in G$ is such that $aX\in \P_{\fin,1}(H_2)$. By the previous lemma, $f$ is a well-defined function. To show that $f$ is bijective, let $Y\in\Pf(H_2)$. Then, by Lemma \ref{lemma}, there is a unique $b\in G$ such that $bY\in\Pf(H_1)$ and we clearly have $f(bY)=b^{-1}(bY)=Y$. Hence $f$ is surjective. Moreover, if $f(X)=aX=Y$ for some $X\in\Pf(H_1)$ and $a\in G$, then $a^{-1}Y=X\in\Pf(H_1)$. Thus, by the uniqueness of $b$, we obtain $b=a^{-1}$ and $X=bY$, which proves the injectivity of $f$. \\
    
    Let now $X,Y\in\P_{\fin,1}(H_1)$. Then $f(X)=aX$, $f(Y)=bY$ and $f(XY)=cXY$ for some $a,b,c\in G$. Since \[
    f(X)f(Y)=abXY\in\Pf(H_2),
    \] it follows by the uniqueness of $c$ that $ab=c$, which proves that $f$ is an isomorphism.
\end{proof}

\begin{corollary}
    There are non-isomorphic, reduced valuation monoids $H_1$ and $H_2$ such that $\P_{\fin,1}(H_1)\simeq \P_{\fin,1}(H_2)$.
\end{corollary}

\begin{proof}
    We define the (additively written) monoids \[H_1:=(\Z\times\mathbb{N})\cup(\mathbb{N}_0\times\{0\})\] and \[H_2:=\{(x,y)\in\Z^2:y\leq \alpha x\},\] where $0<\alpha\in\mathbb{R}\setminus\mathbb{Q}$. 
    It is easy to verify that both $H_1$ and $H_2$ are reduced valuation monoids with $\mathsf{q}(H_1)\simeq\mathsf{q}(H_2)\simeq(\Z^2,+)$. Moreover, the element $(1,0)$ is irreducible in $H_1$, whereas by item B4 of \cite[Theorem 10]{Le22}, $H_2$ does not contain any irreducible elements. Hence $H_1$ and $H_2$ are not isomorphic, and the assertion follows from Theorem \ref{prop}.
\end{proof}

\end{document}